\documentclass[12pt]{amsart}
\usepackage{amssymb}
\usepackage[all]{xy}

\newtheorem{theorem}{Theorem}[section]
\newtheorem{definition}[theorem]{Definition}
\newtheorem{proposition}[theorem]{Proposition}

\begin{document}

\title[Unitary easy quantum groups]{Unitary easy quantum groups: geometric aspects}

\author{Teodor Banica}
\address{T.B.: Department of Mathematics, University of Cergy-Pontoise, F-95000 Cergy-Pontoise, France. {\tt teo.banica@gmail.com}}


\begin{abstract}
We discuss the classification problem for the unitary easy quantum groups, under strong axioms, of noncommutative geometric nature. Our main results concern the intermediate easy quantum groups $O_N\subset G\subset U_N^+$. To any such quantum group we associate its Schur-Weyl twist $\bar{G}$, two noncommutative spheres $S,\bar{S}$, a noncommutative torus $T$, and a quantum reflection group $K$. Studying $(S,\bar{S},T,K,G,\bar{G})$ leads then to some natural axioms, which can be used in order to investigate $G$ itself. We prove that the main examples are covered by our formalism, and we conjecture that in what concerns the case $U_N\subset G\subset U_N^+$, our axioms should restrict the list of known examples.
\end{abstract}

\maketitle

\section*{Introduction}

It is well-known that any compact Lie group appears as a closed subgroup of a unitary group, $G\subset U_N$. As explained by Wang in \cite{wa1}, the unitary group $U_N$ has a free version $U_N^+$, and studying the closed quantum subgroups $G\subset U_N^+$ is a problem of general interest. Such subgroups include the usual compact Lie groups $G\subset U_N$, their $q=-1$ twists, and the duals $G=\widehat{\Gamma}$ of the finitely generated discrete groups $\Gamma=<g_1,\ldots,g_N>$.

The closed subgroups $G\subset U_N^+$ do not have an analogue of a Lie algebra, or much known differential geometric structure. As explained by Woronowicz in \cite{wo1}, \cite{wo2}, this is not an issue. The representation theory of the subgroups $G\subset U_N^+$ can be succesfully developed, with an analogue of the Peter-Weyl theory, of Tannakian duality, and of the Weingarten integration formula. Thus, at least in what concerns certain algebraic geometric and probabilistic aspects, the theory of the subgroups $G\subset U_N^+$ is potentially as powerful as that of the usual subgroups $G\subset U_N$, and just waits to be developed, and applied.

\bigskip

A closed subgroup $G\subset U_N^+$ is called ``easy'' when its Tannakian dual comes in the simplest possible way: from set-theoretic partitions. To be more precise, given a partition $\pi\in P(k,l)$ between an upper row of $k$ points, and a lower row of $l$ points, we can associate to it the following operator, between tensor powers of $\mathbb C^N$:
$$T_\pi(e_{i_1}\otimes\ldots\otimes e_{i_k})=\sum_{j_1\ldots j_l}\delta_\pi\begin{pmatrix}i_1&\ldots&i_k\\ j_1&\ldots&j_l\end{pmatrix}e_{j_1}\otimes\ldots\otimes e_{j_l}$$

Here $\delta_\pi\in\{0,1\}$ is a generalized Kronecker symbol, whose value depends on whether the indices fit or not. Now with this notion in hand, a quantum group $G\subset U_N^+$ is called easy when the following formula holds, for certain subsets $D(k,l)\subset P(k,l)$:
$$Hom(u^{\otimes k},u^{\otimes l})=span\left(T_\pi\Big|\pi\in D(k,l)\right)$$

As a basic example, $O_N$ is easy, due to an old result of Brauer \cite{bra}, with $D=P_2$ being the set of pairings. Easy as well is $S_N$, with $D=P$ being the set of all partitions. When allowing $k,l$ to be colored integers, this formalism covers $U_N$ too, with $D=\mathcal P_2$ being the set of ``matching'' pairings. There are several other examples of type $G\subset U_N$, and all these examples can be ``liberated'' into quantum groups $G^+\subset U_N^+$, by using a standard idea from free probability \cite{bpa}, \cite{vdn}, namely that of removing the crossings from $D$. 

\bigskip

The easy quantum groups, being the ``simplest'' from a Tannakian point of view, are at the center of the compact quantum group theory. Understanding their structure, and finding classification results for them, are fundamental questions.

The easy subgroups $G\subset O_N^+$ were classified by Raum and Weber in \cite{rwe}. In the general case, the situation is considerably more complicated. Of particular interest here is the classification of the intermediate easy quantum groups $U_N\subset U_N^\times\subset U_N^+$.

One idea in dealing with this latter question comes from noncommutative geometry. Roughly speaking, the whole world we are living in, and its geometry, comes from $U_N$. So, in order to understand the partial liberations $U_N^\times$, we should try to ``create'' the associated partial liberated world and geometry. Indeed, in this process we will certainly meet some obstructions, and so we will end up with a finer understanding of $U_N^\times$.

\bigskip

In practice now, let us first think at $U_N$, as an object on its own. The first thing that $U_N$ ``sees'' is the complex sphere $S^{N-1}_\mathbb C$. Indeed, this sphere appears by rotating the point $(1,0,\ldots,0)$ by the matrices $U\in U_N$. Conversely, $S^{N-1}_\mathbb C$ can see $U_N$ too, because this is its isometry group. Thus, we have a correspondence $U_N\leftrightarrow S^{N-1}_\mathbb C$.

This correspondence is already quite interesting for us, but let us further complicate things, by putting into the picture two more basic objects, namely the standard torus $\mathbb T^N$, and its isometry group $K_N=\mathbb T\wr S_N$. It is not difficult to see that we have a full set of $3\times4=12$ correspondences between our 4 objects, so our diagram becomes:
$$\xymatrix@R=40pt@C=40pt{
K_N\ar[r]\ar[d]\ar[dr]&U_N\ar[l]\ar[d]\ar[dl]\\
\mathbb T^N\ar[u]\ar[ur]\ar[r]&S^{N-1}_\mathbb C\ar[l]\ar[ul]\ar[u]
}$$

Now let us further complicate things by twisting at $q=-1$. Since the torus $\mathbb T^N=\widehat{\mathbb Z^N}$ is ultimately a ``discrete'' object, and so is its isometry group $K_N$, these objects are not twistable. Thus, what we have to do is to simply add the twists $\bar{U}_N,\bar{S}^{N-1}_\mathbb C$ of $U_N,S^{N-1}_\mathbb C$. There are now plenty of correspondences between our 6 objects, and by deleting some of them, in order to keep things simple, our diagram becomes as follows:
$$\xymatrix@R=40pt@C=40pt{
\bar{U}_N\ar[r]\ar[d]&K_N\ar[r]\ar[d]\ar[l]&U_N\ar[l]\ar[d]\\
\bar{S}^{N-1}_\mathbb C\ar[u]\ar[r]&\mathbb T^N\ar[u]\ar[r]\ar[l]&S^{N-1}_\mathbb C\ar[l]\ar[u]
}$$

Summarizing, starting with $U_N$, and allowing some basic geometric operations, we ended up with $6$ objects, having correspondences between them. We can stop here, and call the resulting structure, objects + correspondences, ``basic geometry''.

\bigskip

With these preliminaries made, we can go back now to the classification problem for the easy quantum groups. The strategy would be as follows:
\begin{enumerate}
\item Find the intermediate easy quantum groups $U_N\subset U_N^\times\subset U_N^+$ which allow the construction of ``basic geometries'', in the above sense.

\item Intersect these quantum groups $U_N^\times$ with the known examples of free qauntum groups $G_N^+$, as to obtain easy quantum groups $G_N^\times=G_N^+\cap U_N^\times$.
\end{enumerate}

Of course, there is a priori no reason for obtaining all the easy quantum groups $G\subset U_N^+$ in this way. What we should expect from such a program would be rather the axiomatization and classification of a ``core class'' of easy quantum groups $G\subset U_N^+$, which can be extended afterwards, to the general easy setting, and beyond.

\bigskip

In practice now, these ideas have been around for about 5 years, but the progress has been quite slow. Our aim here is comment on the status of this program, notably by continuing our recent work with Bichon \cite{bb1}, \cite{bb2}. We will review the definition of the easy ``noncommutative geometries'' given there, using only the unitary group, the torus and the sphere, by taking into consideration the twisting operation as well.

\bigskip

In order to explain our results, let us first remark that the ``basic geometry'' diagram for $U_N$ has a well-known and important real counterpart, as follows:
$$\xymatrix@R=40pt@C=40pt{
\bar{O}_N\ar[r]\ar[d]&H_N\ar[r]\ar[d]\ar[l]&O_N\ar[l]\ar[d]\\
\bar{S}^{N-1}_\mathbb R\ar[u]\ar[r]&\mathbb Z_2^N\ar[u]\ar[r]\ar[l]&S^{N-1}_\mathbb R\ar[l]\ar[u]
}$$

We can cover this diagram as well, by extending our formalism, and looking at the intermediate easy quantum groups $O_N\subset G\subset U_N^+$. Having such an extension is something quite interesting and useful, because, unlike in the classical case, in the general noncommutative setting the distinction between $\mathbb R,\mathbb C$ becomes ``blurred''. As an example here, the projective versions of $O_N^+,U_N^+$ are known to coincide.

\bigskip

After developing some general theory, we will conclude that we have a number of noncommutative geometries in our sense, having unitary groups as follows:
$$\xymatrix@R=15mm@C=15mm{
U_N\ar[r]&U_N^*\ar[r]&U_N^+\\
\mathbb TO_N\ar[r]\ar[u]&\mathbb TO_N^*\ar[r]\ar[u]&\mathbb TO_N^+\ar[u]\\
O_N\ar[r]\ar[u]&O_N^*\ar[r]\ar[u]&O_N^+\ar[u]}$$

The classification problem for the intermediate liberations $U_N\subset U_N^\times\subset U_N^+$ remains open. There are several natural candidates here, constructed in \cite{bb1}, \cite{bb2}, \cite{bdd}, and we conjecture that our axioms should substantially restrict the list of examples. Yet another approach could come from a more systematic study of the half-liberations of $U_N$, from a physical point of view, in relation with \cite{bdd}, \cite{bd+}, and with the work in \cite{cco}, \cite{con}.

\bigskip

The paper is organized as follows: in 1-2 we review the previous axiomatization work in \cite{bb2}, by taking into account the quantum reflections, in 3-6 we restrict the attention to the easy case, and we develop some general theory here, and in 7-8 we discuss the twisting operation, we present the new axioms and theory, and we discuss open questions.

\bigskip

\noindent {\bf Acknowledgements.} I would like to thank John, Fred, Poufinette and Ursula, for advice and support, during the preparation of the present paper.

\section{Basic axioms}

We review here the material from our recent paper with Bichon \cite{bb2}. The idea there was to axiomatize the triples $(S,T,G)$ which satisfy what we can expect from a noncommutative sphere, a noncommutative torus, and a unitary quantum group. We will refine this formalism, by adding into the picture a quantum reflection group $K$ as well. In other words, we are looking for suitable axioms for quadruplets $(S,T,K,G)$.

In order to talk about noncommutative spheres, we must first introduce the ``biggest'' such sphere, which will contain all the other ones, that we will construct later on. The correct definition for this sphere, from \cite{ba1}, is as follows:

\begin{definition}
The free complex sphere $S^{N-1}_{\mathbb C,+}$ is the compact noncommutative space appearing as abstract spectrum of the following algebra: 
$$C(S^{N-1}_{\mathbb C,+})=C^*\left((x_i)_{i=1,\ldots,N}\Big|\sum_ix_ix_i^*=\sum_ix_i^*x_i=1\right)$$
The closed subspace $S^{N-1}_{\mathbb R,+}\subset S^{N-1}_{\mathbb C,+}$ defined by assuming $x_i=x_i^*$ is called free real sphere.
\end{definition}

Here we use of course some basic knowledge of the $C^*$-algebra theory, and of the Gelfand duality theorem in particular. For some preliminaries here, we refer to \cite{ba1}, \cite{ba2}.

As a first observation, we have inclusions $S^{N-1}_\mathbb C\subset S^{N-1}_{\mathbb C,+}$ and $S^{N-1}_\mathbb R\subset S^{N-1}_{\mathbb R,+}$. It is elementary to check that these inclusions are not isomorphisms. See \cite{ba1}.

Regarding the tori, no special discussion is needed here, because these tori will appear by definition as abstract duals $T=\widehat{\Gamma}$ of certain discrete groups $\Gamma$.

Regarding now the quantum groups, we use here Woronowicz's formalism in \cite{wo1}, \cite{wo2}, under the extra assumption $S^2=id$. We will be particularly interested in Wang's quantum group $U_N^+$, and its orthogonal version $O_N^+$, constructed as follows:

\begin{definition}
The free unitary group $U_N^+$ is the compact quantum group appearing as abstract spectrum of the following algebra, 
$$C(U_N^+)=C^*\left((u_{ij})_{i,j=1,\ldots,N}\Big|u^*=u^{-1},u^t=\bar{u}^{-1}\right)$$
with comultiplication, counit and antipode maps given by:
$$\Delta(u_{ij})=\sum_ku_{ik}\otimes u_{kj}\quad,\quad\varepsilon(u_{ij})=\delta_{ij}\quad,\quad S(u_{ij})=u_{ji}^*$$
The quantum subgroup $O_N^+\subset U_N^+$ defined via $u_{ij}=u_{ij}^*$ is called free orthogonal group.
\end{definition}

Here the fact that the above morphisms $\Delta,\varepsilon,S$  are indeed well-defined follows from the universality property of $C(U_N^+)$, and in view of \cite{wo1}, \cite{wo2}, this allows us to call $U_N^+$ a compact quantum group. A similar discussion applies to $O_N^+$. See \cite{wa1}.

Finally, we will need the free analogues of the hyperoctahedral group $H_N=\mathbb Z_2\wr S_N$, and of its complex version $K_N=\mathbb T\wr S_N$. We have the following definition:

\begin{definition}
The quantum reflection group $K_N^+$ appears as a closed subgroup of $U_N^+$, via the following construction at the algebra level:
$$C(K_N^+)=C(U_N^+)\Big/\Big<u_{ij}u_{ij}^*=u_{ij}^*u_{ij}={\rm magic}\Big>$$ 
The free hyperoctahedral group $H_N^+\subset K_N ^+$ appears by imposing the conditions $u_{ij}=u_{ij}^*$.
\end{definition}

Here the magic condition, due to Wang \cite{wa2}, states that the elements $p_{ij}=u_{ij}u_{ij}^*$ must be projections, summing up to 1 on each row and each column of $p=(p_{ij})$. By using the general theory from \cite{bic}, \cite{wa2} one can deduce that we have free wreath product decompositions of type  $H_N^+=\mathbb Z_2\wr_*S_N^+$ and $K_N^+=\mathbb T\wr_*S_N^+$. See \cite{bbc}, \cite{bve}.

Following now \cite{bb2}, we have the following key notions:

\begin{definition}
Consider a subspace $S\subset S^{N-1}_{\mathbb C,+}$, and a subgroup $G\subset U_N^+$.
\begin{enumerate}
\item The standard torus of $S$ is the subspace $T\subset S$ obtained by setting, at the algebra level, $C(T)=C(S)/<x_ix_i^*=x_i^*x_i=\frac{1}{N}>$.

\item The diagonal torus of $G$ is the subspace $T\subset G$ obtained by setting, at the algebra level, $C(T)=C(G)/<u_{ij}=0|\forall i\neq j>$.
\end{enumerate}
\end{definition}

This definition requires, as usual, some basic $C^*$-algebra knowledge.

Regarding the first construction, observe that the rescaled generators $u_i=\sqrt{N}x_i$ are subject to the relations $u_i^*=u_i^{-1}$, which produce the free group algebra $C^*(F_N)$. Thus, we obtain here a closed subpace of the dual of the free group, $T\subset\widehat{F_N}$.

Regarding now the second construction, observe that the algebra $C(T)$ is generated by the variables $g_i=u_{ii}$, which are group-like. Thus, we obtain here a group dual $T=\widehat{\Gamma}$, where $\Gamma=<g_1,\ldots,g_N>$ is the discrete group generated by these variables.

We refer to \cite{bb2} for full details regarding this material. At the level of basic examples now, the situation is quite interesting, because we have:

\begin{proposition}
The standard and diagonal tori associated to the basic noncommutative spheres and quantum groups are as follows:
\begin{enumerate}
\item $S=S^{N-1}_\mathbb C$, $K=K_N$, $G=U_N$ all produce the standard torus $T=\mathbb T^N$. 

\item $S=S^{N-1}_\mathbb R$, $K=H_N$, $G=O_N$ all produce the standard cube $T=\mathbb Z_2^N$. 

\item $S=S^{N-1}_{\mathbb C,+}$, $K=K_N^+$, $G=U_N^+$ all produce the group dual $T=\widehat{F_N}$. 

\item $S=S^{N-1}_{\mathbb R,+}$, $K=H_N^+$, $G=O_N^+$ all produce the group dual $T=\widehat{\mathbb Z_2^{*N}}$. 
\end{enumerate}
\end{proposition}

\begin{proof}
Here the assertions (1,2) are well-known, and their free analogues (3,4) are well-known too, and follow from the definition of the various objects involved. See \cite{bb2}.
\end{proof}

Summarizing, we have good correspondences $S,K,G\to T$. In order to finish the axiomatization work, the extra piece of theory that we will need is as follows: 

\begin{definition}
Consider an algebraic submanifold $X\subset S^{N-1}_{\mathbb C,+}$, i.e. a closed subset defined via algebraic relations, and a closed quantum subgroup $G\subset U_N^+$.
\begin{enumerate}
\item We say that we have an affine action $G\curvearrowright X$ when the formula $\Phi(x_i)=\sum_ju_{ij}\otimes x_j$ defines a morphism of algebras $\Phi:C(X)\to C(G)\otimes C(X)$.

\item The biggest quantum subgroup $G\subset U_N^+$ acting affinely on $X$ is denoted $G^+(X)$, and is called affine quantum isometry group of $X$.
\end{enumerate}
\end{definition}

Observe that the morphism in (1) is automatically coassociative, $(\Phi\otimes id)\Phi=(id\otimes\Delta)\Phi$, and counital as well, $(id\otimes\varepsilon)\Phi=id$. When $X,G$ are both classical such a morphism must appear by transposition from a usual affine group action $G\times X\to X$.

Regarding now (2), it is routine to check that such a biggest quantum group exists indeed, simply by dividing $C(U_N^+)$ by the appropriate relations. See \cite{ba3}, \cite{gos}.

Following \cite{bb2}, we can now formulate our main definition, as follows:

\begin{definition}
A noncommutative geometry $(S,T,K,G)$ consists of
\begin{enumerate}
\item an intermediate algebraic manifold $S^{N-1}_\mathbb R\subset S\subset S^{N-1}_{\mathbb C,+}$, called sphere,

\item an intermediate compact space $\widehat{\mathbb Z_2^N}\subset T\subset\widehat{F_N}$, called torus,

\item and intermediate quantum groups $H_N\subset K\subset K_N^+$ and $O_N\subset G\subset U_N^+$,
\end{enumerate}
such that the following conditions are satisfied,
\begin{enumerate}
\item $K=G\cap K_N^+$, $G=<K,O_N>$,

\item $T$ is the standard torus of $S$, and the diagonal torus of $K,G$,

\item $G=G^+(S)$, $K=G\cap G^+(T)$,
\end{enumerate}
where we agree to identify the full and reduced quantum group algebras.
\end{definition}

Here we use the standard operations $(G,H)\to G\cap H$ and $(G,H)\to<G,H>$ for the subgroups of $U_N^+$, with the free/reduced convention made at the end. To be more precise, assuming $C(G)=C(U_N^+)/I$ and $C(H)=C(U_N^+)/J$, we set $C(G\cap H)=C(U_N^+)/<I,J>$. Also, assuming that $G,H$ come from Tannakian categories $C,D$, we let $<G,H>$ be the quantum group associated to the Tannakian category $C\cap D$. See \cite{ntu}.

At the level of the main examples, we have the following result:

\begin{theorem}
We have $4$ basic noncommutative geometries, as follows,
$$\xymatrix@R=30pt@C=30pt{
S^{N-1}_\mathbb C\ar[r]&S^{N-1}_{\mathbb C,+}\\
S^{N-1}_\mathbb R\ar[u]\ar[r]&S^{N-1}_{\mathbb R,+}\ar[u]
}\quad\xymatrix@R=20pt@C=10pt{\\ :}\quad
\xymatrix@R=30pt@C=30pt{
\mathbb T^N\ar[r]&\widehat{F_N}\\
\mathbb Z_2^N\ar[u]\ar[r]&\widehat{\mathbb Z_2^{*N}}\ar[u]
}\quad\xymatrix@R=20pt@C=10pt{\\ :}\quad
\xymatrix@R=30pt@C=30pt{
K_N\ar[r]&K_N^+\\
H_N\ar[u]\ar[r]&H_N^+\ar[u]
}\quad\xymatrix@R=20pt@C=10pt{\\ :}\quad
\xymatrix@R=30pt@C=30pt{
U_N\ar[r]&U_N^+\\
O_N\ar[u]\ar[r]&O_N^+\ar[u]
}$$
with the spheres $S$ corresponding to the tori $T$, and to the quantum groups $K,G$.
\end{theorem}

\begin{proof}
This is well-known, with the various axioms corresponding to a number of standard results from the quantum group literature. The idea is as follows:

(1) The axiom $K=G\cap K_N^+$ holds indeed, by definition of the various quantum groups $K$ involved. As for the axiom $G=<K,O_N>$, whose verification involves categories and Tannakian theory, here the results are known from \cite{bbc}, \cite{bve}. 

(2) We know indeed from Proposition 1.5 above that $T$ is the standard torus of $S$, as well as the diagonal torus of $K,G$, in all the 4 cases under investigation.

(3) The axiom $G=G^+(S)$ holds indeed, and for details here, we refer to \cite{ba1}, \cite{bgo}, \cite{bhg}. Regarding now the axiom $K=G\cap G^+(T)$, the verifications here are well-known as well, but use a number of more advanced ingredients. First of all, the quantum isometry groups $G^+(T)$ of the tori were computed in \cite{ba3}, and they are as follows:
$$\xymatrix@R=30pt@C=30pt{
\mathbb T^N\ar[r]&\widehat{F_N}\\
\mathbb Z_2^N\ar[u]\ar[r]&\widehat{\mathbb Z_2^{*N}}\ar[u]
}\qquad\xymatrix@R=20pt@C=10pt{\\ \ar@{~}[r]&}\qquad
\xymatrix@R=30pt@C=30pt{
\bar{U}_N\ar[r]&K_N^+\\
\bar{O}_N\ar[u]\ar[r]&H_N^+\ar[u]
}$$

Here $\bar{O}_N,\bar{U}_N$ are the $q=-1$ twists of the groups $O_N,U_N$, constructed in \cite{ba1}, \cite{bbc}. Now by intersecting with corresponding quantum groups $G$, we obtain, as desired:
$$\xymatrix@R=30pt@C=30pt{
U_N\cap\bar{U}_N\ar[r]&U_N^+\cap K_N^+\\
O_N\cap\bar{O}_N\ar[u]\ar[r]&O_N^+\cap H_N^+\ar[u]
}\qquad\xymatrix@R=20pt@C=10pt{\\ =}\qquad
\xymatrix@R=30pt@C=30pt{
K_N\ar[r]&K_N^+\\
H_N\ar[u]\ar[r]&H_N^+\ar[u]
}$$

Indeed, the results on the left follow from the computations in \cite{ba2}, and the results on the right are trivial, by definition of the quantum groups $K$ involved. 
\end{proof}

\section{Diagrams, easiness}

In the reminder of this paper we basically restrict the attention to the easy case, with the aim of refining our noncommutative geometry axioms, in this case.

Let us begin with some standard quantum group definitions:

\begin{definition}
We call ``colored integer'' a sequence of type $k=\bullet\circ\circ\bullet\circ\ldots\,$ The length of such an integer is the total number of $\circ$ and $\bullet$ symbols, denoted $|k|\in\mathbb N$. Also:
\begin{enumerate}
\item Given a quantum group corepresentation $u$, its tensor powers $u^{\otimes k}$, with $k$ being colored integers, are defined by $u^\circ=u,u^\bullet=\bar{u}$ and multiplicativity. 

\item The Tannakian category associated to $G\subset U_N^+$ is the collection of vector spaces $\mathcal C_G(k,l)=Hom(u^{\otimes k},u^{\otimes l})$, with $k,l$ ranging over the colored integers.
\end{enumerate}
\end{definition}

The terminology here comes from the fact that $\mathcal C_G$ is stable under taking tensor products $\otimes$, compositions $\circ$, and adjoints $*$. It is known from standard Peter-Weyl theory that any irreducible corepresentation of $G$ appears as a subcorepresentation of some $u^{\otimes k}$, with $k$ being a colored integer, and further building on this fact leads to the Tannakian duality for compact quantum groups, which states that $G$ can be fully reconstructed from $\mathcal C_G$. We refer to \cite{wo2} for the original result, and to \cite{mal} for a simplified presentation of it. 

Generally speaking, a closed subgroup $G\subset U_N^+$ is called ``easy'' when $\mathcal C_G$ appears in the simplest possible way, as the span of certain linear maps associated to the set-theoretic partitions. In order for this to hold, we must have in fact $S_N\subset G\subset U_N^+$.

In what follows we are interested only in the notion of easiness for the quantum groups $H_N\subset G\subset U_N^+$, and the combinatorial objects that we will need are as follows:

\begin{definition}
We let $P_{even}(k,l)$ be the set of partitions between an upper row of $|k|$ points and a lower row of $|l|$ points, colored by the $\circ,\bullet$ symbols of $k,l$, having the property that each block contains an even number of elements. Also, we define:
\begin{enumerate}
\item $\mathcal P_{even}(k,l)\subset P_{even}(k,l)$: the partitions which are ``matching'', in the sense that when rotating, each block contains the same number of $\circ$ and $\bullet$ symbols.

\item $\mathcal{NC}_{even}(k,l)\subset NC_{even}(k,l)\subset P_{even}(k,l)$: the set of noncrossing matching partitions, respectively the set of all noncrossing partitions.

\item $\mathcal{NC}_2(k,l)\subset NC_2(k,l),\mathcal P_2(k,l)\subset P_2(k,l)$: the sets of noncrossing matching pairings, noncrossing pairings, matching pairings, and all pairings.
\end{enumerate}
\end{definition} 

Here the rotation operation is by definition the clockwise one, which transforms the partitions $\pi\in P_{even}(k,l)$ into partitions $\eta\in P_{even}(l+k)$, with the convention that when rotating the upper legs, all the $\circ,\bullet$ symbols get reversed, $\circ\leftrightarrow\bullet$.

We denote by $\mathcal{NC}_{even}\subset NC_{even},\mathcal P_{even}\subset P_{even}$ and by $\mathcal{NC}_2\subset NC_2,\mathcal P_2\subset P_2$ the sets formed by the above partitions and pairings, by taking a disjoint union over $k,l$. 

As a first observation, the relation between these sets is as follows:

\begin{proposition}
We have a diagram as follows, with all the arrows being inclusions:
$$
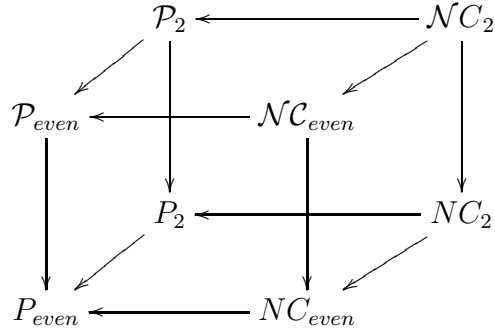
\xymatrix@R=20pt@C=20pt{
&\mathcal P_2\ar[dd]\ar[dl]&&{\mathcal NC}_2\ar[dl]\ar[ll]\ar[dd]\\
\mathcal P_{even}\ar[dd]&&\mathcal{NC}_{even}\ar[ll]\ar[dd]\\
&P_2\ar[dl]&&NC_2\ar[dl]\ar[ll]\\
P_{even}&&NC_{even}\ar[ll]
}$$
In addition, this is an ``intersection diagram'', in the sense that the intersection  of any two objects $X,Y$ appears on the diagram, as the biggest object contained into $X,Y$.
\end{proposition}

\begin{proof}
The first assertion is clear from definitions, and the intersection diagram statement is clear as well, once again starting from the definitions.
\end{proof}

Following \cite{bsp}, \cite{tw1}, let us introduce now the following notion:

\begin{definition}
A set $D=\bigsqcup_{k,l}D(k,l)$ with $D(k,l)\subset P_{even}(k,l)$ is called a category of partitions when it is stable under the following operations:
\begin{enumerate}
\item The horizontal concatenation operation $\otimes$.

\item The vertical concatenation operation $\circ$, performed when the middle colored integers match, with the convention that the resulting closed loops are removed.

\item The upside-down turning operation $*$, with convention that when turning, all the colors are switched, $\circ\leftrightarrow\bullet$.
\end{enumerate}
\end{definition}

As a basic example here, the 8 sets of partitions appearing in Proposition 2.3 above are all categories of partitions. This follows indeed from definitions.

We have now all the needed ingredients for introducing the notion of easiness:

\begin{definition}
A quantum group $H_N\subset G\subset U_N^+$ is called easy when we have
$$Hom(u^{\otimes k},u^{\otimes l})=span\left(T_\pi\Big|\pi\in D(k,l)\right)$$
for any $k,l$, for a certain category of partitions $\mathcal{NC}_2\subset D\subset P_{even}$, where
$$T_\pi(e_{i_1}\otimes\ldots\otimes e_{i_k})=\sum_{j_1\ldots j_l}\delta_\pi\begin{pmatrix}i_1&\ldots&i_k\\ j_1&\ldots&j_l\end{pmatrix}e_{j_1}\otimes\ldots\otimes e_{j_l}$$
with the value of $\delta_\pi\in\{0,1\}$ depending on whether the indices fit or not.
\end{definition}

As already mentioned above, the notion of easiness is in fact a bit more general than this, dealing with quantum groups $S_N\subset G\subset U_N^+$. For the purposes of the present paper, where all relevant quantum groups will contain $H_N$, the above definition is the one that we need. We refer to \cite{bsp}, \cite{fre}, \cite{rwe}, \cite{tw2} for full details regarding the easy quantum group theory, and for \cite{mal} for the needed technical background.

Now back to our geometric considerations, our first task is that of comparing the easiness properties of $G,K$. Let us begin with the following standard result:

\begin{proposition}
For an easy quantum group $H_N\subset G\subset U_N^+$, coming from a category of partitions ${\mathcal NC}_2\subset D\subset P_{even}$, we have:
\begin{eqnarray*}
C(G)&=&C(U_N^+)\Big/\left<T_\pi\in Hom(u^{\otimes k},u^{\otimes l})\Big|\forall k,l,\forall\pi\in D(k,l)\right>\\
D(k,l)&=&\left\{\pi\in P_{even}(k,l)\Big|T_\pi\in Hom(u^{\otimes k},u^{\otimes l})\right\}
\end{eqnarray*}
Moreover, in the first formula we can use if we want only partitions $\pi$ generating $D$.
\end{proposition}

\begin{proof}
All the assertions in the statement follow indeed from the general theory of the easy quantum groups, as explained for instance in \cite{mal}. 
\end{proof}

We will need as well the following result, which is standard as well:

\begin{proposition}
Given two easy quantum groups $H_N\subset G,K\subset U_N^+$, coming respectively from categories of partitions ${\mathcal NC}_2\subset D,E\subset P_{even}$, we have
\begin{eqnarray*}
K=G\cap K_N^+&\iff&E=<D,{\mathcal NC}_{even}>\\
G=<K,O_N>&\iff&D=E\cap P_2
\end{eqnarray*}
where $\cap$ and $<,>$ are the usual intersection operation, and object generated by operation, for the closed subgroups of $U_N^+$, and for the subcategories of $P_{even}$.
\end{proposition}

\begin{proof}
This follows indeed from Proposition 2.6 above, and from the definition of the operations $\cap$ and $<,>$ in the quantum group case.
\end{proof}

We have now all the needed ingredients for formulating:

\begin{proposition}
For a geometry $(S,T,K,G)$, the following are equivalent:
\begin{enumerate}
\item $G$ is easy, coming from a category of pairings ${\mathcal NC}_2\subset D\subset P_2$.

\item $K$ is easy, coming from a category of partitions ${\mathcal NC}_{even}\subset E\subset P_{even}$.
\end{enumerate} 
If these conditions are satisfied, the categories $D,E$ must satisfy the conditions $D=E\cap P_2$ and $E=<D,{\mathcal NC}_{even}>$, and we call our geometry easy. 
\end{proposition}

\begin{proof}
This follows from Proposition 2.7 above, which shows that when $G$ is easy, then so is $K=G\cap K_N^+$, and that when $K$ is easy, then so is $G=<K,O_N>$, with the corresponding categories of partitions $D,E$ being given by the formulae there.
\end{proof}

At the level of basic examples we have the following result, whose formulation for the unitary quantum groups only goes back to \cite{bb2}: 

\begin{theorem}
The basic $4$ geometries are all easy. More precisely, the corresponding quantum groups $K,G$ form the following intersection diagram,
$$
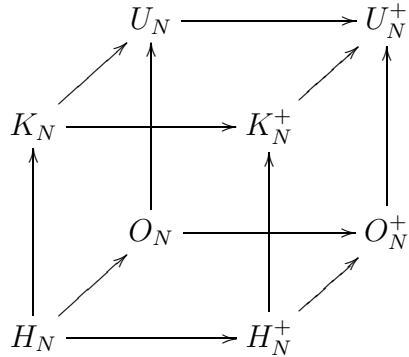
\xymatrix@R=20pt@C=20pt{
&U_N\ar[rr]&&U_N^+\\
K_N\ar[rr]\ar[ur]&&K_N^+\ar[ur]\\
&O_N\ar[rr]\ar[uu]&&O_N^+\ar[uu]\\
H_N\ar[uu]\ar[ur]\ar[rr]&&H_N^+\ar[uu]\ar[ur]
}$$
and the corresponding categories of partitions are those in Proposition 2.3 above.
\end{theorem}

\begin{proof}
The fact that we have indeed an interesection diagram follows from definitions, with the relation between the $K,G$ groups coming from the axiom $K=G\cap U_N^+$. 

Regarding now the easiness assertion, here the results for $O_N,U_N$ are well-known, due to Brauer \cite{bra}, the results for $O_N^+,U_N^+$ are well-known free versions of Brauer's theorem, and all this material is explained in detail in \cite{bb2}. As for the results for the corresponding quantum reflection groups, these are known from \cite{bbc}, \cite{bve}.
\end{proof}

Summarizing, we have now an extension of the formalism in \cite{bb2}, taking into account the quantum reflection groups as well, and covering the main 4 examples. 

We should mention that our axioms here do not cover a whole number of interesting examples of noncommutative spheres, such as those considered in \cite{cdu}, \cite{cla}, \cite{ddl}. However, the theory of easy quantum groups, and so our present formalism as well, are quite flexible, as explained for instance in \cite{cwe}, \cite{fre}, \cite{tw2}, and so the question of gradually extending the considerations here makes sense, and might solve some of these problems.

\section{Main examples}

In order to construct further examples, the idea is to look for intermediate objects for the inclusions $U_N\subset U_N^+$ and $O_N^+\subset U_N^+$, and then to take intersections. There are several possible choices here, and the ``standard'' choices lead to the following result:

\begin{proposition}
We have intersection diagrams of easy quantum groups as follows,
$$\xymatrix@R=15mm@C=15mm{
K_N\ar[r]&K_N^*\ar[r]&K_N^+\\
\mathbb TH_N\ar[r]\ar[u]&\mathbb TH_N^*\ar[r]\ar[u]&\mathbb TH_N^+\ar[u]\\
H_N\ar[r]\ar[u]&H_N^*\ar[r]\ar[u]&H_N^+\ar[u]}
\qquad\xymatrix@R=18mm@C=5mm{\\:&\\&\\}\ 
\xymatrix@R=15mm@C=15mm{
U_N\ar[r]&U_N^*\ar[r]&U_N^+\\
\mathbb TO_N\ar[r]\ar[u]&\mathbb TO_N^*\ar[r]\ar[u]&\mathbb TO_N^+\ar[u]\\
O_N\ar[r]\ar[u]&O_N^*\ar[r]\ar[u]&O_N^+\ar[u]}$$
with $U_N^*\subset U_N^+$ coming via the relations $abc=cba$, with $a,b,c\in\{u_{ij},u_{ij}^*\}$, with $\mathbb TO_N^+\subset U_N^+$ coming via the relations $ab^*=a^*b$, with $a,b\in\{u_{ij}\}$, and with the quantum groups on the left being obtained from those on the right, by intersecting with $K_N^+$.
\end{proposition}

\begin{proof}
The relations $abc=cba$ with $a,b,c\in\{u_{ij},u_{ij}^*\}$, and $ab^*=a^*b$  with $a,b\in\{u_{ij}\}$, correspond respectively to the following two diagrams:
$$
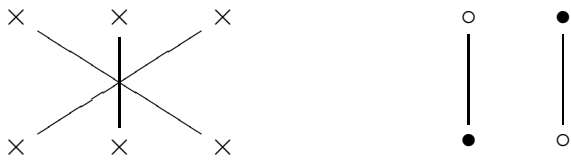
\xymatrix@R=12mm@C=8mm{
\times\ar@{-}[drr]&\times\ar@{-}[d]&\times\ar@{-}[dll]
\\
\times&\times&\times}\qquad\ \qquad\ \qquad
\xymatrix@R=12mm@C=8mm{
\circ\ar@{-}[d]&\bullet\ar@{-}[d]
\\
\bullet&\circ}$$

Here the $\times$ symbols stand for the fact that we can use there any combinations of the symbols $\circ$ and $\bullet$, provided that the resulting pairing is matching.

We conclude that $U_N^*,\mathbb TO_N^+\subset U_N^+$ are both easy quantum groups, and by taking intersections, as indicated, we obtain 9+9 easy quantum groups, as stated.
\end{proof}

Regarding now the associated categories of partitions, we have:

\begin{proposition}
The categories of partitions and pairings for the above $9+9$ easy quantum groups are given by the following intersection diagrams,
$$\xymatrix@R=15mm@C=11mm{
\mathcal P_{even}\ar[d]&\mathcal P_{even}^*\ar[l]\ar[d]&\mathcal{NC}_{even}\ar[l]\ar[d]\\
\bar{P}_{even}\ar[d]&\bar{P}_{even}^*\ar[l]\ar[d]&\bar{NC}_{even}\ar[l]\ar[d]\\
P_{even}&P_{even}^*\ar[l]&NC_{even}\ar[l]}
\ \quad\xymatrix@R=18mm@C=5mm{\\ :&\\&\\}\ 
\xymatrix@R=15mm@C=16mm{
\mathcal P_2\ar[d]&\mathcal P_2^*\ar[l]\ar[d]&\mathcal{NC}_2\ar[l]\ar[d]\\
\bar{P}_2\ar[d]&\bar{P}_2^*\ar[l]\ar[d]&\bar{NC}_2\ar[l]\ar[d]\\
P_2&P_2^*\ar[l]&NC_2\ar[l]}$$
with the categories $P_{even}^*,\bar{P}_{even}\subset P_{even}$ consisting of the partitions which when flattened, have the same number of $\circ,\bullet$ symbols in each block, when relabelling the legs $\circ\bullet\circ\bullet\ldots$, respectively which have the same number of $\circ,\bullet$ symbols, and with the categories on the right being obtained from those on the left by intersecting with $P_2$.
\end{proposition}

\begin{proof}
As explained in Theorem 2.9 above, the results at the 4 corners hold indeed. Regarding the remaining computations, these are well-known as well. For the diagram on the right, the result here is from \cite{bb2}, the idea being as follows:

(1) Regarding the intermediate quantum group $O_N\subset O_N^*\subset O_N^+$, here the defining relations, namely $[ab,cd]=0$, correspond indeed to the category $NC_2\subset P_2^*\subset P_2$ constructed in the statement. This is something well-known, explained in \cite{bb2}.

(2) Regarding the intermediate quantum group $O_N\subset \mathbb TO_N\subset U_N$, here the defining relations, namely $ab^*=a^*b$, correspond indeed to the category $\mathcal P_2\subset\bar{P}_2^*\subset P_2$ constructed in the statement. Once again this is something well-known, explained in \cite{bb2}.

(3) Finally, regarding the various intersections, and notably the quantum group $U_N^*$ from \cite{bdu}, the axioms hold here as well, due to the various results in \cite{ba1}, \cite{bb2}. Indeed, the diagrams for $U_N^*$ are well-known to be those for $O_N^*$, with matching colorings.

Regarding the diagram on the left, here the results are known from \cite{ba2}, \cite{ba3}, \cite{rwe}. We can proceed as well directly, by using our axioms. Indeed, we know that the quantum reflection groups are given by $K=G\cap K_N^+$. Thus Proposition 2.7 applies, and shows that in each of the cases under investigation, $K$ is easy, having $E=<D,{\mathcal NC}_{even}>$ as associated category of partitions. But this gives the categories in the statement.
\end{proof}

We are now in position of stating our main result, as follows:

\begin{theorem}
We have $9$ main easy geometries, with spheres and tori as follows,
$$\xymatrix@R=13mm@C=11mm{
S^{N-1}_\mathbb C\ar[r]&S^{N-1}_{\mathbb C,*}\ar[r]&S^{N-1}_{\mathbb C,+}\\
\mathbb TS^{N-1}_\mathbb R\ar[r]\ar[u]&\mathbb TS^{N-1}_{\mathbb R,*}\ar[r]\ar[u]&\mathbb TS^{N-1}_{\mathbb R,+}\ar[u]\\
S^{N-1}_\mathbb R\ar[r]\ar[u]&S^{N-1}_{\mathbb R,*}\ar[r]\ar[u]&S^{N-1}_{\mathbb R,+}\ar[u]}
\ \ \ \ \ \xymatrix@R=18mm@C=5mm{\\ :&\\&\\}\ \xymatrix@R=13mm@C=13mm{
\mathbb T^N\ar[r]&\widehat{\mathbb Z^{\circ N}}\ar[r]&\widehat{F_N}\\
\mathbb T\mathbb Z_2^N\ar[r]\ar[u]&\mathbb T\widehat{\mathbb Z_2^{\circ N}}\ar[r]\ar[u]&\mathbb T\widehat{\mathbb Z_2^{*N}}\ar[u]\\
\mathbb Z_2^N\ar[r]\ar[u]&\widehat{\mathbb Z_2^{\circ N}}\ar[r]\ar[u]&\widehat{\mathbb Z_2^{*N}}\ar[u]}$$
with the various middle objects coming via the relations $abc=cba$ and $ab^*=a^*b$, imposed to the standard coordinates, and to their adjoints.
\end{theorem}

\begin{proof}
The $4$ results at the corners are already known, from Theorem 1.8 above. Regarding the remaining results, these are basically known since \cite{bb2}, with the exception of the verification of the $K=G\cap G^+(T)$ axiom, that we will perform now.

According to the various results in \cite{bb2} and \cite{ba3}, the quantum isometry groups $G^+(T)$ of the tori in the statement are given by:
$$\xymatrix@R=14mm@C=12mm{
\mathbb T^N\ar[r]&\widehat{\mathbb Z^{\circ N}}\ar[r]&\widehat{F_N}\\
\mathbb T\mathbb Z_2^N\ar[r]\ar[u]&\mathbb T\widehat{\mathbb Z_2^{\circ N}}\ar[r]\ar[u]&\mathbb T\widehat{\mathbb Z_2^{*N}}\ar[u]\\
\mathbb Z_2^N\ar[r]\ar[u]&\widehat{\mathbb Z_2^{\circ N}}\ar[r]\ar[u]&\widehat{\mathbb Z_2^{*N}}\ar[u]}
\ \ \ \ \ \xymatrix@R=18mm@C=5mm{\\ :&\\&\\}\ 
\xymatrix@R=15mm@C=13mm{
\bar{U}_N\ar[r]&\bar{U}_N^*\ar[r]&K_N^+\\
\mathbb T\bar{O}_N\ar[r]\ar[u]&\mathbb T\bar{O}_N^*\ar[r]\ar[u]&\mathbb TH_N^+\ar[u]\\
\bar{O}_N\ar[r]\ar[u]&\bar{O}_N^*\ar[r]\ar[u]&H_N^+\ar[u]}$$

Now by intersecting with the corresponding quantum groups $G$, and by using the various computations in \cite{ba2}, we obtain the following quantum groups:
$$\xymatrix@R=15mm@C=6mm{
U_N\cap\bar{U}_N\ar[r]&U_N^*\cap\bar{U}_N^*\ar[r]&U_N^+\cap K_N^+\\
\mathbb TO_N\cap\mathbb T\bar{O}_N\ar[r]\ar[u]&\mathbb TO_N^*\cap\mathbb T\bar{O}_N^*\ar[r]\ar[u]&\mathbb TO_N^+\cap \mathbb TH_N^+\ar[u]\\
O_N\cap\bar{O}_N\ar[r]\ar[u]&O_N^*\cap\bar{O}_N^*\ar[r]\ar[u]&O_N^+\cap H_N^+\ar[u]}
\ \ \ \xymatrix@R=18mm@C=5mm{\\ =&\\&\\}
\xymatrix@R=15mm@C=10mm{
K_N\ar[r]&K_N^*\ar[r]&K_N^+\\
\mathbb TH_N\ar[r]\ar[u]&\mathbb TH_N^*\ar[r]\ar[u]&\mathbb TH_N^+\ar[u]\\
H_N\ar[r]\ar[u]&H_N^*\ar[r]\ar[u]&H_N^+\ar[u]}$$

Thus, we have checked the extra axiom, and we are done.
\end{proof}

\section{Ergodicity issues}

The axioms that we have so far, concerning the quadruplets $(S,T,K,G)$, already provide us with a number of correspondences between the objects involved. To be more precise, the diagram of the known correspondences is as follows:
$$\xymatrix@R=50pt@C=60pt{
K\ar[r]\ar[d]&G\ar[l]\ar[dl]\\
T&S\ar[l]\ar[ul]\ar[u]
}$$

In order to have a complete diagram, we would need for instance arrows $T\to G$ and $G\to S$. The question of finding an arrow $T\to G$ will be discussed in the next section. As for the question of finding an arrow $G\to S$, we will discuss this issue here.

In the classical cases, the spheres naturally appear as homogeneous spaces, $O_N\to S^{N-1}_\mathbb R$ and $U_N\to S^{N-1}_\mathbb C$. So, let us begin with the following construction:

\begin{proposition}
Given a subgroup $G\subset U_N^+$, let $G\to S$ be first column space of $G$, in the sense that $C(S)\subset C(G)$ is the algebra generated by the variables $x_i=u_{i1}$.
\begin{enumerate}
\item We have an embedding $S\subset S^{N-1}_{\mathbb C,+}$.

\item We have an affine action $G\curvearrowright S$.
\end{enumerate}
In addition, for $G=O_N,U_N$ we obtain in this way the spheres $S=S^{N-1}_\mathbb R,S^{N-1}_\mathbb C$.
\end{proposition}

\begin{proof}
The fact that we have an embedding $S\subset S^{N-1}_{\mathbb C,+}$ comes from the biunitarity condition satisfied by the fundamental corepresentation $u=(u_{ij})$, as follows:
\begin{eqnarray*}
\sum_ix_i^*x_i&=&\sum_iu_{i1}^*u_{i1}=(u^*u)_{11}=1\\
\sum_ix_ix_i^*&=&\sum_iu_{i1}u_{i1}^*=(u^t\bar{u})_{11}=1
\end{eqnarray*}

Regarding now the action $G\curvearrowright S$, this simply appears as the restriction of the standard action $G\curvearrowright G$. To be more precise, the coaction $\Phi:C(S)\to C(G)\otimes C(S)$, which must map $x_i\to\sum_ku_{ik}\otimes x_k$, appears as restriction of the comultiplication $\Delta:C(G)\to C(G)\otimes C(G)$, which maps $u_{ij}\to\sum_ku_{ik}\otimes u_{kj}$, and so $u_{i1}\to\sum_ku_{ik}\otimes u_{k1}$.

Finally, the last assertion is well-known, and elementary.
\end{proof}

In the free case now, as explained in \cite{ba1}, \cite{bgo}, the spheres $S^{N-1}_{\mathbb R,+}$ and $S^{N-1}_{\mathbb C,+}$ do not exactly appear as first column spaces, and this due to some analytic issues, coming from the well-known ``full vs. reduced version'' problem with Gelfand duality.

We can, however, formulate a positive result, as follows:

\begin{proposition}
Let $(S,T,K,G)$ be a noncommutative geometry.
\begin{enumerate}
\item We have an inclusion $S_{min}\subset S$, between subspaces of $S^{N-1}_{\mathbb C,+}$, where $G\to S_{min}$ is the first column space of $G$.

\item The composition $\int_S:C(S)\to C(S_{min})\subset C(G)\to\mathbb C$, with the map on the right being the Haar integration over $G$, is $G$-equivariant.

\item The inclusion $S_{min}\subset S$ becomes an equality, when identifying full and reduced algebras, with respect to the null ideals of the corresponding $\int_S$ forms.
\end{enumerate}
\end{proposition}

\begin{proof}
These results are all elementary, as follows:

(1) Given a noncommutative geometry $(S,T,K,G)$ we have $S^{N-1}_\mathbb R\subset S$, and so an evaluation morphism $ev_1:C(S)\to\mathbb C$, given by $x_i\to\delta_{i1}$. We also have the canonical action $G\curvearrowright S$. We therefore obtain a factorization, as claimed:
$$\xymatrix@R=12mm@C=18mm{
C(S)\ar@.[r]\ar[d]_\Phi&C(S_{min})\ar[d]^i\\
C(G)\otimes C(S)\ar[r]^{id\otimes ev_1}&C(G)
}
\qquad\xymatrix@R=5mm@C=10mm{\\:}\qquad
\xymatrix@R=12mm@C=18mm{
x_i\ar@.[r]\ar[d]_\Phi&X_i\ar[d]^i\\
\sum_ju_{ij}\otimes x_j\ar[r]^{id\otimes ev_1}&u_{i1}
}$$

(2) By functoriality, the $G$-equivariance of $\int_S$ only needs to be verified for $S=S_{min}$. But here $\int_S$ simply appears as the restriction of $\int_G$, and the result follows.

(3) This is something trivial, corresponding to the fact that the canoncial surjection $C(S)\to C(S_{min})$ becomes an isomorphism, when performing the GNS construction with respect to the corresponding $\int_S$ forms on these two algebras.
\end{proof}

Summarizing, we have an answer to our question, stating that $G\to S$ appears, when identifying ``full and reduced algebras'', simply by taking the first column space.

This answer is of course something quite theoretical, and in practice, we would need some finer results on the subject. One way of dealing with this problem is as follows:

\begin{definition}
A noncommutative geometry $(S,T,K,G)$ is called ergodic when
$$\xymatrix@R=12mm@C=20mm{
C(S)\ar[r]^\Phi\ar[d]_\alpha&C(G)\otimes C(S)\ar[d]^{\int_G\otimes id}\\
C(G)\ar[r]^{\int_G(.)1}&C(S)
}$$
commutes, where $\alpha:C(S)\to C(S_{min})\subset C(G)$ is the canonical map.
\end{definition}

As a basic example, assuming that we are in the case $S=S_{min}$, the ergodicity condition is automatic. Indeed, in this case $\alpha:C(S)\subset C(G)$ is the canonical inclusion, and the ergodicity condition follows from the invariance axiom for the Haar measure, namely:
$$\left(\int_G\otimes id\right)\Delta=\int_G(.)1$$

In general, there is no reason for an arbitrary noncommutative geometry $(S,T,K,G)$, as axiomatized in Definition 1.7 above, to be automatically ergodic. We have:

\begin{proposition}
The ergodicity condition is equivalent to the following condition,
$$\sum_{j_1\ldots j_p}\int_Gu_{i_1j_1}^{e_1}\ldots u_{i_pj_p}^{e_p}\cdot x_{j_1}^{e_1}\ldots x_{j_p}^{e_p}=\int_Gu_{i_11}^{e_1}\ldots u_{i_p1}^{e_p}\cdot1$$
which must be valid for any indices $i_1,\ldots,i_p$, and exponents $e_1,\ldots,e_p$.
\end{proposition}

\begin{proof}
We recall from the proof of Proposition 4.2 above that the canonical map $\alpha:C(S)\to C(S_{min})\subset C(G)$ has the following alternative description:
$$\alpha=(id\otimes ev_1)\Phi$$

By using this formula we obtain, by composing with the integration over $G$:
$$\int_G\alpha=\int_G(id\otimes ev_1)\Phi=ev_1\left(\int_G\otimes id\right)\Phi$$

Thus, with $I=(\int_G\otimes id)\Phi$, the ergodicity condition from Definition 4.3 reads:
$$I=ev_1I(.)1$$

On the other hand, the values of $I$ are given by:
\begin{eqnarray*}
I(x_{i_1}^{e_1}\ldots x_{i_p}^{e_p})
&=&\left(\int_G\otimes id\right)\left(\sum_{j_1\ldots j_p}u_{i_1j_1}^{e_1}\ldots u_{i_pj_p}^{e_p}\otimes x_{j_1}^{e_1}\ldots x_{j_p}^{e_p}\right)\\
&=&\sum_{j_1\ldots j_p}\int_Gu_{i_1j_1}^{e_1}\ldots u_{i_pj_p}^{e_p}\cdot x_{j_1}^{e_1}\ldots x_{j_p}^{e_p}
\end{eqnarray*}

Thus we obtain the condition in the statement, and we are done.
\end{proof}

We can now complement Proposition 4.2 with a finer result, as follows:

\begin{theorem}
Let $(S,T,K,G)$ be a noncommutative geometry, assumed to be ergodic, in the above sense. We have then inclusions as follows,
$$S_{min}\subset S\subset S_{max}$$
where $G\to S_{min}$ is the first column space of $G$, and where $S_{max}$ is the space obtained from $S^{N-1}_{\mathbb C,+}$ by imposing the conditions found in Proposition 4.4 above.
\end{theorem}

\begin{proof}
This is indeed clear from Definition 4.3 and Proposition 4.4.
\end{proof}

Summarizing, in order to have a good correspondence $G\to S$, we must impose the somewhat technical ergodicity condition found above. We will further discuss these topics in sections 5-6 below, with a refinement of the ergodicity condition, in the easy case, and with a proof as well for the fact that the 9 main geometries are indeed ergodic.

\section{Fizzy geometries}

Our purpose here is to fine-tune the easy geometry formalism. We would like to extract from what we have a notion of ``fully easy'' (or ``fizzy'') geometry, where everything is axiomatized directly in terms of the underlying categories of partitions.

For an easy geometry $(S,T,K,G)$, coming from categories of partitions and pairings $(E,D)$, the diagram of the known correspondences is as follows:
$$\xymatrix@R=20pt@C=40pt{
K\ar[rr]\ar[dd]\ar[dr]&&G\ar[ll]\ar[dl]\\
&E\leftrightarrow D\ar[ul]\ar[ur]\ar[dl]\\
T&&S\ar[ll]\ar[ul]\ar[uu]
}$$

Our first task will be that of working out the formula of $T$, in terms of $(E,D)$. Then, under suitable conditions coming from section 4 above, we will do the same for $S$.

Regarding the torus $T$, we first have the following result, from \cite{bb2}:

\begin{proposition}
Given an easy quantum group $H_N\subset G\subset U_N^+$, coming from a category of partitions ${\mathcal NC}_2\subset D\subset P_{even}$, the associated diagonal torus is given by:
$$C(T)=C^*(F_N)\Big/\left<g_{i_1}\ldots g_{i_k}=g_{j_1}\ldots g_{j_l}\Big|\forall i,j,k,l,\exists\pi\in D(k,l),\delta_\pi\begin{pmatrix}i\\ j\end{pmatrix}\neq0\right>$$
Moreover, we can use if we want only partitions $\pi$ generating $D$.
\end{proposition}

\begin{proof}
We prove directly the second assertion. Assume that $S\subset D$ is a generating subset. According to Proposition 2.6 above, we have the following formula:
$$C(G)=C(U_N^+)\Big/\left<T_\pi\in Hom(u^{\otimes k},u^{\otimes l})\Big|\forall k,l,\forall\pi\in S(k,l)\right>$$

In order to compute the diagonal torus algebra $C(T)$, we must further divide by the relations $u_{ij}=0$ for $i\neq j$. But this gives the following formula, where $g=diag(g_1,\ldots,g_N)$ is the diagonal matrix formed by the standard generators of $F_N$:
$$C(T)=C^*(F_N)\Big/\left<T_\pi\in Hom(g^{\otimes k},g^{\otimes l})\Big|\forall k,l,\forall\pi\in S(k,l)\right>$$

In order to interpret now the above relations, observe that we have:
\begin{eqnarray*}
T_\pi g^{\otimes k}(e_{i_1}\otimes\ldots\otimes e_{i_k})&=&\sum_{j_1\ldots j_l}\delta_\pi\begin{pmatrix}i_1&\ldots&i_k\\ j_1&\ldots&j_l\end{pmatrix}e_{j_1}\otimes\ldots\otimes e_{j_l}\cdot g_{i_1}\ldots g_{i_k}\\
g^{\otimes l}T_\pi(e_{i_1}\otimes\ldots\otimes e_{i_k})&=&\sum_{j_1\ldots j_l}\delta_\pi\begin{pmatrix}i_1&\ldots&i_k\\ j_1&\ldots&j_l\end{pmatrix}e_{j_1}\otimes\ldots\otimes e_{j_l}\cdot g_{j_1}\ldots g_{j_l}
\end{eqnarray*}

Thus, we are led to the formula in the statement. See \cite{bb2}.
\end{proof}

In the noncommutative geometric setting now, we have:

\begin{proposition}
Given categories ${\mathcal NC}_2\subset D\subset P_2$ and ${\mathcal NC}_{even}\subset E\subset P_{even}$ satisfying $D=E\cap P_2$, $E=<D,{\mathcal NC}_{even}>$, the following formulae define the same algebra:
\begin{eqnarray*}
C(T)&=&C^*(F_N)\Big/\left<g_{i_1}\ldots g_{i_k}=g_{j_1}\ldots g_{j_l}\Big|\forall i,j,k,l,\exists\pi\in D(k,l),\delta_\pi\begin{pmatrix}i\\ j\end{pmatrix}\neq0\right>\\
C(T)&=&C^*(F_N)\Big/\left<g_{i_1}\ldots g_{i_k}=g_{j_1}\ldots g_{j_l}\Big|\forall i,j,k,l,\exists\pi\in E(k,l),\delta_\pi\begin{pmatrix}i\\ j\end{pmatrix}\neq0\right>
\end{eqnarray*}
Moreover, the space $T$ defined by these formulae is the common diagonal torus of the easy quantum groups $O_N\subset G\subset U_N^+$ and $H_N\subset K\subset K_N^+$ associated to $D,E$.
\end{proposition}

\begin{proof}
All the assertions follow by applying Proposition 5.1, first for the quantum group $G$, and then for the quantum group $K$, by using our assumptions on $D,E$.
\end{proof}

Observe that the conditions in Proposition 5.2 are satisfied for an easy geometry $(S,T,K,G)$, coming from categories $(D,E)$. Thus, we have our needed formula for $T$ in terms of both $D,E$, and we can go ahead with a similar study for the sphere $S$.

Regarding the spheres, the situation here is more complicated. Let us go back to the ergodicity considerations from section 4 above. With standard Weingarten integration theory conventions, as in \cite{bsp}, we have the following result:

\begin{proposition}
An easy geometry $(S,T,K,G)$, coming from a category of pairings ${\mathcal NC}_2\subset D\subset P_2$, is ergodic precisely when we have
$$\sum_{\pi,\sigma\in D}\delta_\pi(i_1,\ldots,i_p)W(\pi,\sigma)\sum_{j_1\ldots j_p}\delta_\sigma(j_1\ldots j_p)x_{j_1}^{e_1}\ldots x_{j_p}^{e_p}=\sum_{\pi,\sigma\in D}\delta_\pi(i_1,\ldots,i_p)W(\pi,\sigma)$$
for any $i_1,\ldots,i_p$ and any $e_1,\ldots,e_p$, where $W$ is the Weingarten function of $G$.
\end{proposition}

\begin{proof}
We know from Proposition 4.4 above that the ergodicity condition reads:
$$\sum_{j_1\ldots j_p}\int_Gu_{i_1j_1}^{e_1}\ldots u_{i_pj_p}^{e_p}\cdot x_{j_1}^{e_1}\ldots x_{j_p}^{e_p}=\int_Gu_{i_11}^{e_1}\ldots u_{i_p1}^{e_p}\cdot1$$

On the other hand, the Weingarten integration formula for $G$ is as follows:
$$\int_Gu_{i_1j_1}^{e_1}\ldots u_{i_pj_p}^{e_p}=\sum_{\pi,\sigma\in D}\delta_\pi(i_1,\ldots,i_p)\delta_\sigma(j_1\ldots j_p)W(\pi,\sigma)$$

Thus, we obtain the formula in the statement.
\end{proof}

Now observe that the ergodicity condition, in its above combinatorial formulation, trivially holds when the following condition is satisfied over $S$, for any choice of the indices $j_1,\ldots,j_p$, of the exponents $e_1,\ldots,e_p$, and of a pairing $\pi\in D$:
$$\sum_{j_1\ldots j_p}\delta_\sigma(j_1\ldots j_p)x_{j_1}^{e_1}\ldots x_{j_p}^{e_p}=1$$

Let us look now more in detail at this condition. By definition, this is a kind of ``strong ergodicity'' condition on the action $G\curvearrowright S$. The point now is that, as we will show below, this condition is satisfied for all the basic 9 examples of easy geometries. In addition, once again for these basic 9 examples, the above relations present in fact the spheres $S$.

Summarizing, our ergodicity study led us to exactly what we wanted, namely a concrete procedure for constructing a noncommutative sphere $S$, out of a category of pairings $D$. Thus, back to our axiomatization questions, we are led to the following notion:

\begin{definition}
An easy noncommutative geometry $(S,T,K,G)$, coming from a category of pairings ${\mathcal NC}_2\subset D\subset P_2$, is called fully easy (or fizzy) when we have
$$\sum_{j_1\ldots j_p}\delta_\sigma(j_1\ldots j_p)x_{j_1}^{e_1}\ldots x_{j_p}^{e_p}=1$$
for any $j_1,\ldots,j_p$, any $e_1,\ldots,e_p$, and any $\sigma\in D$, over the sphere $S$.
\end{definition}

Observe that, due to Proposition 5.3, the fizziness condition implies, trivially, that the ergodicity condition is satisfied. Observe also that in the case of a fizzy geometry we have $S_{min}\subset S\subset S_{med}$, where $S_{med}\subset S^{N-1}_{\mathbb C,+}$ is the subspace obtained by imposing the easiness condition. Thus, we have here an improvement of the result in Theorem 4.5.

As a main result now, we can axiomatize the fizzy geometries, directly in terms of the category of pairings $D$. With the standard notations $D(p)=D(0,p)$ for the categories of partitions, and $Fix(r)=Hom(r,1)$ for corepresentations, the result is as follows:

\begin{theorem}
For a fizzy geometry $(S,T,K,G)$, coming from a category of pairings ${\mathcal NC}_2\subset D\subset P_2$, we have, by identifying as usual free and reduced algebras:
\begin{eqnarray*}
C(S)&=&C(S^{N-1}_{\mathbb C,+})\Big/\left<\sum_{i_1\ldots i_p}\delta_\pi(i_1\ldots i_p)x_{i_1}^{e_1}\ldots x_{i_p}^{e_p}=1\Big|\forall p,\forall\pi\in D(p)\right>\\
C(T)&=&C^*(F_N)\Big/\left<g_{i_1}^{e_1}\ldots g_{i_p}^{e_p}=1\Big|\forall p,\forall i,\exists\pi\in D(p),\delta_\pi(i_1\ldots i_p)\neq0\right>\\
C(K)&=&C(K_N^+)\Big/\left<T_\pi\in Fix(u^{\otimes p})\Big|\forall p,\forall\pi\in D(p)\right>\\
C(G)&=&C(U_N^+)\Big/\left<T_\pi\in Fix(u^{\otimes p})\Big|\forall p,\forall\pi\in D(p)\right>
\end{eqnarray*}
Conversely, given two categories ${\mathcal NC}_2\subset D\subset P_2$ and ${\mathcal NC}_{even}\subset E\subset P_{even}$ satisfying $D=E\cap P_2$ and $E=<D,\mathcal{NC}_{even}>$, the objects $(S,T,K,G)$ defined above form a noncommutative geometry precisely when $G=G^+(S)$ and $K=G\cap G^+(T)$.
\end{theorem}

\begin{proof}
The first assertion follows from the various results obtained above, by using the fact that the collection of sets $D_0=(D(p))$ generates the category $D=(D(k,l))$. 

Regarding the last assertion, let us go back to the axioms (1,2,3) in Definition 1.7:

(1) This axiom is satisfied, due to our combinatorial assumptions on $D,E$, namely $D=E\cap P_2$ and $E=<D,\mathcal{NC}_{even}>$, as explained in Proposition 2.8 above.

(2) We know from Proposition 5.2 that $T$ is the common diagonal torus of $K,G$, so we just have to check that $T$ coincides with the standard torus $\mathcal T$ of the sphere $S$. In order to compute this latter torus, observe that, by definition of $S$, we have:
$$C(\mathcal T)=C^*(F_N)\Big/\left<\sum_{i_1\ldots i_p}\delta_\pi(i_1\ldots i_p)x_{i_1}^{e_1}\ldots x_{i_p}^{e_p}=1\Big|\forall p,\forall\pi\in D(p)\right>$$

Now by rescaling, the corresponding group generators $g_i=\sqrt{N}x_i$ are subject to the following relations, which must be valid in the group algebra $C(\mathcal T)$:
$$\sum_{i_1\ldots i_p}\delta_\pi(i_1\ldots i_p)g_{i_1}^{e_1}\ldots g_{i_p}^{e_p}=p\quad,\quad\forall p,\forall\pi\in D(p)$$

But these are exactly the relations defining the dual $\Gamma=\widehat{T}$, and we are done.

(3) Here the axioms hold as well, because they are included in the statement.
\end{proof}

\section{Examples, revised}

We have seen in the previous section that the easy geometry formalism can be fine-tuned into a fully easy (or fizzy) geometry formalism, where everything is axiomatized as it should be, in terms of the associated categories of pairings and partitions $D,E$.

In this section we verify that the 9 main geometries are indeed fizzy. In order to check the fizziness axiom, from Definition 5.4 above, we use the following fact:

\begin{proposition}
For a permutation $\sigma\in S_p$, the following conditions, involving self-adjoint variables $x_1,\ldots,x_p$ satisfying $\sum_ix_ix_i^*=\sum_ix_i^*x_i=1$, are equivalent,
\begin{enumerate}
\item $x_{i_1}^{e_1}\ldots x_{i_p}^{e_p}=x_{i_{\sigma(1)}}^{f_1}\ldots x_{i_{\sigma(p)}}^{f_p}$,

\item $\sum_{i_1\ldots i_p}x_{i_1}^{e_1}\ldots x_{i_p}^{e_p}x_{i_{\sigma(p)}}^{\bar{f}_p}\ldots x_{i_{\sigma(1)}}^{\bar{f}_1}=1$,
\end{enumerate} 
for any choice of $e_1,\ldots,e_p$ and $f_1,\ldots,f_p$, where $e\to\bar{e}$ is the involution $1\leftrightarrow *$.
\end{proposition}

\begin{proof}
This kind of result is well-known, with the proof being as follows:

$(1)\implies(2)$ By conjugating the equality (1), and then by using $p$ times our assumption $\sum_ix_ix_i^*=\sum_ix_i^*x_i=1$ on the variables $x_i$, we obtain, as desired:
$$\sum_{i_1\ldots i_p}x_{i_1}^{e_1}\ldots x_{i_p}^{e_p}x_{i_{\sigma(p)}}^{\bar{f}_p}\ldots x_{i_{\sigma(1)}}^{\bar{f}_1}
=\sum_{i_1\ldots i_p}x_{i_1}^{e_1}\ldots x_{i_p}^{e_p}x_{i_p}^{\bar{e}_p}\ldots x_{i_1}^{\bar{e}_1}=1$$

$(2)\implies(1)$ The proof here is more complicated, and requires a positivity trick. First, by conjugating the equality (2) we obtain the following formula:
$$\sum_{i_1\ldots i_p}x_{i_{\sigma(1)}}^{f_1}\ldots x_{i_{\sigma(p)}}^{f_p}x_{i_p}^{\bar{e}_p}\ldots x_{i_1}^{\bar{e}_1}=1$$

The point now is that we have the following equality:
\begin{eqnarray*}
&&\sum_{i_1\ldots i_p}(x_{i_1}^{e_1}\ldots x_{i_p}^{e_p}-x_{i_{\sigma(1)}}^{f_1}\ldots x_{i_{\sigma(p)}}^{f_p})(x_{i_1}^{e_1}\ldots x_{i_p}^{e_p}-x_{i_{\sigma(1)}}^{f_1}\ldots x_{i_{\sigma(p)}}^{f_p})^*\\
&=&\sum_{i_1\ldots i_p}x_{i_1}^{e_1}\ldots x_{i_p}^{e_p}x_{i_p}^{\bar{e}_p}\ldots x_{i_1}^{\bar{e}_1}-x_{i_1}^{e_1}\ldots x_{i_p}^{e_p}x_{i_{\sigma(p)}}^{\bar{f}_p}\ldots x_{i_{\sigma(1)}}^{\bar{f}_1}\\
&&-x_{i_{\sigma(1)}}^{f_1}\ldots x_{i_{\sigma(p)}}^{f_p}x_{i_p}^{\bar{e}_p}\ldots x_{i_1}^{\bar{e}_1}+x_{i_{\sigma(1)}}^{f_1}\ldots x_{i_{\sigma(p)}}^{f_p}x_{i_{\sigma(p)}}^{\bar{f}_p}\ldots x_{i_{\sigma(1)}}^{\bar{f}_1}\\
&=&1-1-1+1=0
\end{eqnarray*}

Here we have used $p$ times our assumption $\sum_ix_ix_i^*=\sum_ix_i^*x_i=1$ on the variables $x_i$, for computing the first and the last sums. As for the two middle sums, these were computed by using our assumption (2), and its conjugate.

Now since we are in a situation of type $\sum_IA_IA_I^*=0$, by positivity it follows that we have $A_I=0$ for any $I$, and we recover precisely the condition (1), as desired.
\end{proof}

In order to exploit the above result, let us formulate:

\begin{definition}
Associated to any pairing $\pi\in P_2(k,l)$ are the following relations $\mathcal R_\pi$,
$$\sum_{i_1\ldots i_k}\delta_\sigma\binom{i_1\ldots i_k}{j_1\ldots j_l}x_{i_1}^{e_1}\ldots x_{i_k}^{e_k}=x_{j_1}^{f_1}\ldots x_{j_l}^{f_l}\quad,\quad\forall j_1,\ldots,j_l$$
between the standard coordinates $x_1,\ldots,x_N$ on the sphere $S^{N-1}_{\mathbb C,+}$.
\end{definition}

Also given a pairing $\pi\in P_2(k,l)$, let us denote by $\pi^-\in P_2(l\bar{k})$ the pairing obtained by flattening, with the usual convention that the upper indices get conjugated in this way, and so that the colored integer $k$ gets replaced by its conjugate $\bar{k}$.

The result found in Proposition 6.1 tells us that for a permutation $\pi\in S_p$, the relations $\mathcal R_\pi$ and $\mathcal R_{\pi^-}$ are equivalent. We can further build on this fact, as follows:

\begin{proposition}
Given a category ${\mathcal NC}_2\subset D\subset P_2$, generated by a certain set of permutations $S$, the following relations define the same subspace $S\subset S^{N-1}_{\mathbb C,+}$:
\begin{enumerate}
\item $\mathcal R_\pi$, with $\pi\in S$.

\item $\mathcal R_\pi$, with $\pi\in D$.

\item $\mathcal R_{\pi^-}$, with $\pi\in S$.

\item $\mathcal R_{\pi^-}$, with $\pi\in D$.
\end{enumerate}
\end{proposition}

\begin{proof}
As pointed out above, the equivalence between the relations (1) and (3) follows from Proposition 6.1. As for the equivalences between the relations (1) and (2), and between the relations (3) and (4), these follow from the fact that $S$ generates $D$.
\end{proof}

In practice now, we are interested in the equivalence $(2)\iff(4)$ coming from the above result. So, as a final conclusion to these considerations, let us formulate:

\begin{proposition}
Given a category of pairings ${\mathcal NC}_2\subset D\subset P_2$, we have
\begin{eqnarray*}
&&C(S^{N-1}_{\mathbb C,+})\Big/\left<\sum_{i_1\ldots i_p}\delta_\pi(i_1\ldots i_p)x_{i_1}^{e_1}\ldots x_{i_p}^{e_p}=1\Big|\forall p,\forall\pi\in D(p)\right>
\\
&=&C(S^{N-1}_{\mathbb C,+})\Big/\left<\sum_{i_1\ldots i_p}\delta_\sigma\binom{i_1\ldots i_p}{j_1\ldots j_p}x_{i_1}^{e_1}\ldots x_{i_p}^{e_p}=x_{j_1}^{f_1}\ldots x_{j_p}^{f_p}\Big|\forall p,\forall\sigma\in D(p,p)\right>
\end{eqnarray*}
provided that $D$ is generated by certain permutations.
\end{proposition}

\begin{proof}
This follows indeed from the equivalence $(2)\iff(4)$ from Proposition 6.3.
\end{proof}

Now by using this criterion, we can go back to the 9 spheres, and we obtain:

\begin{theorem}
The basic $9$ geometries are all fizzy, in the sense of Definition 5.4.
\end{theorem}

\begin{proof}
We already know from Theorem 3.3 above that the basic 9 geometries are easy. Thus, we just have to check the fizziness axiom from Definition 5.4 above.

In order to do so, we can use Proposition 6.4 above, which makes the link between the defining relations for the spheres, which involve certain permutations $\pi\in S_p$, and the fizziness condition, which is expressed in terms of certain pairings $\pi\in D(p)$.

To be more precise, let us look at $O_N^+,\mathbb TO_N^+,U_N,U_N^*$, which produce by taking intersections all the 9 basic easy quantum groups. These 4 quantum groups appear by imposing to the standard coordinates of $U_N^+$ the relations coming from the following diagrams:
$$
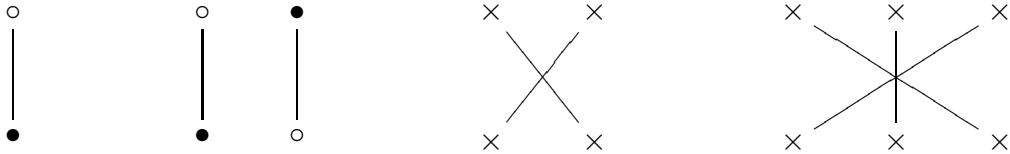
\xymatrix@R=12mm@C=8mm{
\circ\ar@{-}[d]\\
\bullet}
\qquad\qquad\quad
\xymatrix@R=12mm@C=8mm{
\circ\ar@{-}[d]&\bullet\ar@{-}[d]\\
\bullet&\circ}
\qquad\qquad\quad
\xymatrix@R=12mm@C=8mm{
\times\ar@{-}[dr]&\times\ar@{-}[dl]\\
\times&\times}
\qquad\qquad\quad
\xymatrix@R=12mm@C=8mm{
\times\ar@{-}[drr]&\times\ar@{-}[d]&\times\ar@{-}[dll]
\\
\times&\times&\times}$$

Here the various $\times$ symbols stand, as usual, for the fact that we can use there any combinations of symbols $\circ$ and $\bullet$, provided that the resulting pairing is matching. Now since all these diagrams are permutations, Proposition 6.4 applies and gives the result.
\end{proof}

\section{Schur-Weyl twisting}

We are now in position of starting the twisting work, as part of the general program explained in the introduction. We begin with some preliminaries, from \cite{ba1}.

Given $\pi\in P(k,l)$, we can always switch pairs of neighbors, belonging to different blocks, either in the upper row, or in the lower row, as to make $\pi$ noncrossing. 

In the particular case $\pi\in Perm(k,k)$, this is the same as writing $\pi$ as a product of transpositions, and so, computing the signature $\varepsilon(\pi)$. In general now, we have:

\begin{proposition}
We have a signature map $\varepsilon:P_{even}\to\{-1,1\}$, given by $\varepsilon(\pi)=(-1)^c$, where $c$ is the number of switches needed to make $\pi$ noncrossing. In addition:
\begin{enumerate}
\item For $\pi\in Perm(k,k)\simeq S_k$, this is the usual signature.

\item For $\pi\in P_2$ we have $(-1)^c$, where $c$ is the number of crossings.

\item For $\pi\in P$ obtained from $\sigma\in NC_{even}$ by merging blocks, the signature is $1$.
\end{enumerate}
\end{proposition}

\begin{proof}
The fact that the number $c$ in the statement is well-defined modulo 2 is standard, and we refer here to \cite{ba1}. As for the remaining assertions, these are as well from \cite{ba1}:

(1) For $\pi\in Perm(k,k)$ the standard form is $\pi'=id$, and the passage $\pi\to id$ comes by composing with a number of transpositions, which gives the signature. 

(2) For a general $\pi\in P_2$, the standard form is of type $\pi'=|\ldots|^{\cup\ldots\cup}_{\cap\ldots\cap}$, and the passage $\pi\to\pi'$ requires $c$ mod 2 switches, where $c$ is the number of crossings. 

(3) For a partition $\pi\in P_{even}$ coming from $\sigma\in NC_{even}$ by merging a certain number $n$ of blocks, the fact that the signature is 1 follows by recurrence on $n$.
\end{proof}

We can make act partitions in $P_{even}$ on tensors in a twisted way, as follows:

\begin{definition}
Associated to any partition $\pi\in P_{even}(k,l)$ is the linear map
$$\bar{T}_\pi(e_{i_1}\otimes\ldots\otimes e_{i_k})=\sum_{\sigma\leq\pi}\varepsilon(\sigma)\sum_{j:\ker(^i_j)=\sigma}e_{j_1}\otimes\ldots\otimes e_{j_l}$$
where $\varepsilon:P_{even}\to\{-1,1\}$ is the signature map.
\end{definition}

Observe the similarity with the formula in Definition 2.5.

We will need as well the following result, which is also from \cite{ba1}:

\begin{proposition}
The assignement $\pi\to\bar{T}_\pi$ is categorical, in the sense that
$$\bar{T}_\pi\otimes\bar{T}_\sigma=\bar{T}_{[\pi\sigma]}\quad,\quad\bar{T}_\pi\bar{T}_\sigma=N^{c(\pi,\sigma)}\bar{T}_{[^\sigma_\pi]}\quad,\quad\bar{T}_\pi^*=\bar{T}_{\pi^*}$$
where $c(\pi,\sigma)$ is the number of closed loops obtained when composing.
\end{proposition}

\begin{proof}
This follows as in \cite{bsp}, \cite{tw2}, by performing modifications where needed, using Proposition 7.1. To be more precise, our three properties require:
$$\varepsilon\left(\ker\begin{pmatrix}i_1&\ldots&i_p\\ j_1&\ldots&j_q\end{pmatrix}\right)
\varepsilon\left(\ker\begin{pmatrix}k_1&\ldots&k_r\\ l_1&\ldots&l_s\end{pmatrix}\right)=
\varepsilon\left(\ker\begin{pmatrix}i_1&\ldots&i_p&k_1&\ldots&k_r\\ j_1&\ldots &j_q&l_1&\ldots&l_s\end{pmatrix}\right)$$
$$\varepsilon\left(\ker\begin{pmatrix}i_1&\ldots&i_p\\ j_1&\ldots&j_q\end{pmatrix}\right)
\varepsilon\left(\ker\begin{pmatrix}j_1&\ldots&j_q\\ k_1&\ldots&k_r\end{pmatrix}\right)
=\varepsilon\left(\ker\begin{pmatrix}i_1&\ldots&i_p\\ k_1&\ldots&k_r\end{pmatrix}\right)$$
$$\varepsilon\left(\ker\begin{pmatrix}i_1&\ldots&i_p\\ j_1&\ldots&j_q\end{pmatrix}\right)=\varepsilon\left(\ker\begin{pmatrix}j_1&\ldots&j_q\\ i_1&\ldots&i_p\end{pmatrix}\right)$$

But all these formulae can be proved by using the rules (1,2,3) in Proposition 7.1 above, and this gives the result. For full details here, we refer to \cite{ba1}. 
\end{proof}

Finally, we have the following definition, from \cite{ba1}:

\begin{definition}
Given an intermediate easy quantum group $H_N\subset G\subset U_N^+$, where $H_N$ is the hyperoctahedral group, we construct an intermediate quantum group 
$$H_N\subset\bar{G}\subset U_N^+$$
by replacing, at the Tannakian level, the maps $T_\pi$ by their twisted versions $\bar{T}_\pi$. 
\end{definition}

Now back to our noncommutative geometry questions, we are first interested in the 9+9 quantum groups from Theorem 3.3. The corresponding twists are as follows:

\begin{proposition}
The twists of the basic $9+9$ easy quantum groups are
$$\xymatrix@R=13mm@C=13mm{
K_N\ar[r]&K_N^*\ar[r]&K_N^+\\
\mathbb TH_N\ar[r]\ar[u]&\mathbb TH_N^*\ar[r]\ar[u]&\mathbb TH_N^+\ar[u]\\
H_N\ar[r]\ar[u]&H_N^*\ar[r]\ar[u]&H_N^+\ar[u]}
\qquad\xymatrix@R=17mm@C=5mm{\\ :&\\&\\}\ 
\xymatrix@R=13mm@C=13mm{
\bar{U}_N\ar[r]&\bar{U}_N^*\ar[r]&U_N^+\\
\mathbb T\bar{O}_N\ar[r]\ar[u]&\mathbb T\bar{O}_N^*\ar[r]\ar[u]&\mathbb TO_N^+\ar[u]\\
\bar{O}_N\ar[r]\ar[u]&\bar{O}_N^*\ar[r]\ar[u]&O_N^+\ar[u]}$$
where $\bar{U}_N,\bar{U}_N^*$ appear by suitably twisting the defining relations for $U_N,U_N^*$, and where the other quantum groups appear as intersections, according to the above diagram.
\end{proposition}

\begin{proof}
It is known from \cite{ba1}, as a consequence of Proposition 7.1 (3), that the free quantum groups $H_N^+\subset G\subset U_N^+$ equal their own twists. Thus, the quantum groups on the right are the correct ones. Regarding now $\bar{U}_N,\bar{U}_N^*$, these appear indeed as indicated, and for the details here, we refer to \cite{ba1}. Finally, the intersection claims are elementary.

Regarding now the quantum reflection groups, we must prove that all of them are self-dual, $K=\bar{K}$. But this follows from the various results in \cite{ba2}.
\end{proof}

Now back to the case of an arbitrary easy geometry $(S,T,K,G)$, the next step is that of twisting the torus $T$. However, there is no need for a new definition here, because the tori are automatically self-dual, due to the following general result:

\begin{proposition}
The diagonal tori of $G,\bar{G}$ coincide, their duals being given by
$$\Gamma=F_N\Big/\left<g_{i_1}\ldots g_{i_k}=g_{i_1}\ldots g_{j_l}\Big|\forall i,j,k,l,\exists\pi\in D(k,l),\delta_\pi\begin{pmatrix}i\\ j\end{pmatrix}\neq0\right>$$
where $D$ is the associated category of pairings.
\end{proposition}

\begin{proof}
This follows by adapting the proof of Proposition 5.1 above. To be more precise, the proof there works with the $\delta$ symbols replaced by $\bar{\delta}$ symbols, and since we have $\delta\neq0\iff\bar{\delta}\neq0$, the twisted and untwisted computations give the same result.
\end{proof}

Regarding now the spheres, the twisting here goes as follows:

\begin{proposition}
Given a category $\mathcal{NC}_2\subset D\subset P_2$ generated by permutations, the following two sets of relations produce the same noncommutative sphere $\bar{S}\subset S^{N-1}_{\mathbb C,+}$:
\begin{enumerate}
\item $\sum_{j_1\ldots j_p}\bar{\delta}_\pi(j_1\ldots j_p)x_{j_1}^{e_1}\ldots x_{j_p}^{e_p}=1$, with $p\in\mathbb N$ and $\pi\in P(p)$.

\item $\sum_{i_1\ldots i_p}\bar{\delta}_\pi\binom{i_1\ldots i_p}{j_1\ldots j_p}x_{i_1}^{e_1}\ldots x_{i_p}^{e_p}=x_{j_1}^{f_1}\ldots x_{j_p}^{f_p}$, with $p\in\mathbb N$ and $\pi\in P(p,p)$.
\end{enumerate}
In the case of the $9$ basic geometries, $\bar{S}$ appears as a suitable twist of $S$, with the defining relations for $\bar{S}$ being obtained from those for $S$ by adding suitable signs.
\end{proposition}

\begin{proof}
Here the equivalence follows by using the same computation as in the untwisted case, and this equivalence gives the second assertion as well, by performing a case-by-case analysis. All the twisted spheres are in fact already known, from \cite{ba1}.
\end{proof}

As a main result now, basically collecting what we have so far, and including as well a number of new computations, we have:

\begin{theorem}
Given a fizzy geometry $(S,T,K,G)$, we can twist $S,G$, by setting
\begin{eqnarray*}
C(\bar{S})&=&C(S^{N-1}_{\mathbb C,+})\Big/\left<\sum_{i_1\ldots i_p}\bar{\delta}_\pi(i_1\ldots i_p)x_{i_1}^{e_1}\ldots x_{i_p}^{e_p}=1\Big|\forall p,\forall\pi\in D(p)\right>\\
C(\bar{G})&=&C(U_N^+)\Big/\left<\bar{T}_\pi\in Fix(u^{\otimes p})\Big|\forall p,\forall\pi\in D(p)\right>
\end{eqnarray*}
and in the $9$ main cases, we have $K=\bar{K}$ and $\bar{G}=G^+(\bar{S})$.
\end{theorem}

\begin{proof}
The fact that the self-duality condition holds indeed is already known, from Proposition 7.5. As for the twisted isometry axiom, here the results are known from \cite{ba1}.
\end{proof}

As a partial conclusion here, at least in the case of the 9 main geometries, we have objects $(\bar{S},T,K,\bar{G})$, which form a kind of ``noncommutative geometry'' as well.

\section{Axioms, refined}

In this section we go one step further into our axiomatization program, by taking into accounts the twists, as to reach to our final formalism, in this paper. Let us begin with a straightforward definition, obtained by ``twisting'' Definition 1.7 above:

\begin{definition}
A twisted noncommutative geometry $(\bar{S},T,K,\bar{G})$ consists of
\begin{enumerate}
\item an intermediate algebraic manifold $\bar{S}^{N-1}_\mathbb R\subset \bar{S}\subset S^{N-1}_{\mathbb C,+}$, called sphere,

\item an intermediate compact space $\widehat{\mathbb Z_2^N}\subset T\subset\widehat{F_N}$, called torus,

\item and intermediate quantum groups $H_N\subset K\subset K_N^+$ and $\bar{O}_N\subset\bar{G}\subset U_N^+$,
\end{enumerate}
such that the following conditions are satisfied,
\begin{enumerate}
\item $K=\bar{G}\cap K_N^+$, $\bar{G}=<K,\bar{O}_N>$,

\item $T$ is the standard torus of $\bar{S}$, and the diagonal torus of $K,\bar{G}$,

\item $\bar{G}=G^+(\bar{S})$, $K\subset\bar{G}\cap G^+(T)$,
\end{enumerate}
where we agree to identify the full and reduced quantum group algebras.
\end{definition}

All this is in fact quite clumsy and incomplete, because we have no twisted analogue of the key axiom $K=G\cap G^+(T)$ from the untwisted case. In addition, all the examples that we have come in fact by twisting a usual noncommutative geometry $(S,T,K,G)$.

Thus, in order to truly advance, we should rather axiomatize directly the sextuplets of type $(S,\bar{S},T,K,G,\bar{G})$. We are led in this way to the following notion:

\begin{definition}
A bi-easy (or ``busy'') noncommutative geometry is a sextuplet of type $(S,\bar{S},T,K,G,\bar{G})$, with $(S,T,K,G)$ being a fizzy geometry, satisfying: 
\begin{enumerate}
\item Self-duality: $K=\bar{K}$.

\item Twisted isometries: $\bar{G}=G^+(\bar{S})$.
\end{enumerate}
Equivalently, $(\bar{S},T,K,\bar{G})$ must be a twisted geometry, in the above sense.
\end{definition}

Here the equivalence is clear, with the only non-trivial verification, namely that of the self-duality condition $K=\bar{K}$, coming from $K=G\cap K_N^+$ and $K=\bar{G}\cap K_N^+$.

Let us first comment on the above axioms. If we put together all the conditions that we have on $(S,\bar{S},T,K,G,\bar{G})$ we are led to a quite long list, and the point is that some of these conditions are in fact automatic. We have indeed the following result:

\begin{proposition}
Given categories ${\mathcal NC}_{even}\subset E\subset P_{even}$ and ${\mathcal NC}_2\subset D\subset P_2$ satisfying $D=E\cap P_2$ and $E=<D,\mathcal{NC}_{even}>$, the following conditions are automatic:
\begin{enumerate}
\item $G\subset G^+(S)$.

\item $\bar{G}\subset G^+(\bar{S})$.

\item $K\subset G^+(T)$.
\end{enumerate}
In addition, the standard torus of $\bar{S}$ coincides with the standard torus of $S$, and with the common diagonal torus of the quantum groups $K,G,\bar{G}$ as well.
\end{proposition}

\begin{proof}
Regarding the various quantum isometry group assertions, in all cases, we must construct a morphism mapping the variables $x_i$ to the variables $X_i=\sum_ju_{ij}\otimes x_j$. In the untwisted case, we have the following computation, valid for any partition $\pi$:
\begin{eqnarray*}
\sum_{i_1\ldots i_p}\delta_\pi(i_1\ldots i_p)X_{i_1}^{e_1}\ldots X_{i_p}^{e_p}
&=&\sum_{i_1\ldots i_p}\delta_\pi(i_1\ldots i_p)\sum_{j_1\ldots j_p}u_{i_1j_1}^{e_1}\ldots u_{i_pj_p}^{e_p}\otimes x_{j_1}^{e_1}\ldots x_{j_p}^{e_p}\\
&=&\sum_{j_1\ldots j_p}\left(\sum_{i_1\ldots i_p}\delta_\pi(i_1\ldots i_p)u_{i_1j_1}^{e_1}\ldots u_{i_pj_p}^{e_p}\right)\otimes x_{j_1}^{e_1}\ldots x_{j_p}^{e_p}\\
&=&\sum_{j_1\ldots j_p}\delta_\pi(j_1\ldots j_p)x_{j_1}^{e_1}\ldots x_{j_p}^{e_p}\\
&=&1
\end{eqnarray*}

Thus in all cases we have a coaction map given by $x_i\to X_i$, as desired. The proof in the twisted case is similar. Finally, the computation of the standard torus of $\bar{S}$ is similar to the one in the untwisted case, and we obtain the same space.
\end{proof}

Now back to the general case, everything depends on the underlying categories of partitions, and we have the following refimenent of Theorem 5.5 above:

\begin{proposition}
For a busy geometry $(S,\bar{S},T,K,G,\bar{G})$, coming from categories of pairings ${\mathcal NC}_2\subset D\subset P_2$ and of partitions ${\mathcal NC}_{even}\subset E\subset P_{even}$, we have:
\begin{eqnarray*}
C(\dot{S})&=&C(S^{N-1}_{\mathbb C,+})\Big/\left<\sum_{i_1\ldots i_p}\dot{\delta}_\pi(i_1\ldots i_p)x_{i_1}^{e_1}\ldots x_{i_p}^{e_p}=1\Big|\forall p,\forall\pi\in D(p)\right>\\
C(T)&=&C^*(F_N)\Big/\left<g_{i_1}^{e_1}\ldots g_{i_p}^{e_p}=1\Big|\forall p,\forall i,\exists\pi\in D(p),\delta_\pi(i_1\ldots i_p)\neq0\right>\\
C(K)&=&C(K_N^+)\Big/\left<T_\pi\in Fix(u^{\otimes p})\Big|\forall p,\forall\pi\in D(p)\right>\\
C(\dot{G})&=&C(U_N^+)\Big/\left<\dot{T}_\pi\in Fix(u^{\otimes p})\Big|\forall p,\forall\pi\in D(p)\right>
\end{eqnarray*}
Conversely, given two categories ${\mathcal NC}_2\subset D\subset P_2$ and ${\mathcal NC}_{even}\subset E\subset P_{even}$ satisfying $D=E\cap P_2$ and $E=<D,\mathcal{NC}_{even}>$, the objects defined above form a noncommutative geometry when $K=\bar{K}$ and $G=G^+(S)$, $\bar{G}=G^+(\bar{S})$, $K=G\cap G^+(T)$.
\end{proposition}

\begin{proof}
This follows indeed from Theorem 5.5, and from the various results above, which show that the twists $\bar{S},\bar{G}$ appear via formulae which are similar to those for $S,G$.
\end{proof}

At the level of main examples now, we have:

\begin{theorem}
The basic $9$ geometries are all busy.
\end{theorem}

\begin{proof}
This follows indeed from Theorem 6.5 and Theorem 7.8 above.
\end{proof}

The classification problem for the busy geometries remains open. The classical and free cases can be investigated by using \cite{tw1}, \cite{tw2}, and this is probably quite routine. As for the purely unitary case, where $U_N\subset G\subset U_N^+$, our conjecture here is that our present axioms should substantially restrict the list of known examples.

\end{document}